\documentclass[12pt, twoside]{article}
\usepackage{amsmath,amsthm,amssymb}
\usepackage{times}
\usepackage{enumerate}

\def\titlerunning#1{\gdef\titrun{#1}}
\makeatletter
\def\author#1{\gdef\autrun{\def\and{\unskip, }#1}\gdef\@author{#1}}
\def\address#1{{\def\and{\\\hspace*{18pt}}\renewcommand{\thefootnote}{}%
\footnote {#1}}%
\markboth{\autrun}{\titrun}}
\makeatother
\def\email#1{e-mail: #1}
\def\subjclass#1{{\renewcommand{\thefootnote}{}%
\footnote{\emph{Mathematics Subject Classification (2010):} #1}}}
\def\keywords#1{\par\medskip
\noindent\textbf{Keywords.} #1}

\newtheorem{thm}[equation]{Theorem}
\newtheorem{cor}[equation]{Corollary}
\newtheorem{lem}[equation]{Lemma}

\newtheorem{prop}[equation]{Proposition}
\newtheorem{conj}[equation]{Conjecture}

\theoremstyle{definition}
\newtheorem{defn}[equation]{Definition}
\newtheorem{rmk}[equation]{Remark}

\numberwithin{equation}{section}

\newcommand{\BC}{{\mathbb {C}}}

\newcommand{\BZ}{{\mathbb {Z}}}

\newcommand{\bJ}{{\mathbf {J}}}
\newcommand{\bK}{{\mathbf {K}}}
\newcommand{\bU}{{\mathbf {U}}}

\newcommand{\CC}{{\mathcal {C}}}

\newcommand{\CJ}{{\mathcal {J}}}
\newcommand{\CK}{{\mathcal {K}}}

\newcommand{\CM}{{\mathcal {M}}}
\newcommand{\CN}{{\mathcal {N}}}

\newcommand{\CU}{{\mathcal {U}}}

\newcommand{\CW}{{\mathcal {W}}}

\newcommand{\CZ}{{\mathcal {Z}}}

\newcommand{\FA}{{\mathfrak {A}}}
\newcommand{\FB}{{\mathfrak {B}}}

\newcommand{\FF}{{\mathfrak {F}}}

\newcommand{\FP}{{\mathfrak {P}}}

\newcommand{\Fk}{{\mathfrak {k}}}

\newcommand{\Fo}{{\mathfrak {o}}}
\newcommand{\Fp}{{\mathfrak {p}}}

\newcommand{\RB}{{\mathrm {B}}}

\newcommand{\RG}{{\mathrm {G}}}

\newcommand{\RI}{{\mathrm {I}}}

\newcommand{\RL}{{\mathrm {L}}}

\newcommand{\RP}{{\mathrm {P}}}
\newcommand{\RQ}{{\mathrm {Q}}}

\newcommand{\RU}{{\mathrm {U}}}

\newcommand{\RZ}{{\mathrm {Z}}}

\newcommand{\Aut}{{\mathrm{Aut}}}

\newcommand{\End}{{\mathrm{End}}}

\newcommand{\GL}{{\mathrm{GL}}}

\newcommand{\Hom}{{\mathrm{Hom}}}

\newcommand{\Ind}{{\mathrm{Ind}}}
\newcommand{\cInd}{{\text{\rm c-Ind}}}

\newcommand{\Ker}{{\mathrm{Ker}}}

\newcommand{\Mat}{{\mathrm{Mat}}}

\renewcommand{\Re}{{\mathrm{Re}}}

\newcommand{\Supp}{{\mathrm{Supp}}}

\def\bks{{\backslash}}

\def\tilpi{{\widetilde{\pi}}}

\def\tiltau{{\widetilde{\tau}}}

\def\le{\leqslant}
\def\ge{\geqslant}

\begin{document}


\baselineskip=17pt


\titlerunning{The Local Converse Problem for $\GL_n$}

\title{Towards the Jacquet Conjecture on the Local Converse Problem for $p$-adic $\GL_n$}

\author{Dihua Jiang
\and 
Chufeng Nien
\and 
Shaun Stevens}

\date{}

\maketitle

\address{D. Jiang: School of Mathematics,
University of Minnesota,
Minneapolis, MN 55455, USA; \email{dhjiang@math.umn.edu}
\and
C. Nien: Department of Mathematics, National Cheng Kung University and
National Center for Theoretical Sciences (South), Tainan 701, Taiwan; \email{nienpig@mail.ncku.edu.tw}
\and
S. Stevens: School of Mathematics,
University of East Anglia,
Norwich Research Park,
Norwich, NR4 7TJ, UK; \email{Shaun.Stevens@uea.ac.uk}}



\begin{abstract}
The Local Converse Problem is to determine how the family of the 
local gamma factors~$\gamma(s,\pi\times\tau,\psi)$ characterizes the 
isomorphism class of an irreducible admissible generic representation~$\pi$ 
of~$\GL_n(F)$, with~$F$ a non-archimedean local field, where $\tau$ runs through 
all irreducible supercuspidal representations of $\GL_r(F)$ and $r$ runs
through positive integers. The Jacquet conjecture asserts that it is enough to 
take~$r=1,2,\ldots,\left[\frac{n}{2}\right]$. Based on arguments in the work of Henniart 
and of Chen giving preliminary steps towards the Jacquet conjecture,
we formulate a general approach to prove the Jacquet conjecture. With
this approach, the Jacquet conjecture is proved under an assumption
which is then verified in several cases, including the case of level
zero representations.

\end{abstract}

\section{Introduction}

Let~$\pi$ be an irreducible (admissible) generic representation
of~$\RG_n:=\GL_n(F)$, where~$F$ is a 
locally compact non-archimedean local field. 
We may assume that~$n\ge 2$, since the discussion 
in this paper for $n=1$ is trivial. 
Attached to~$\pi$ is the family of local gamma 
factors~$\gamma(s,\pi\times\tau,\psi)$, with~$\tau$ any irreducible generic
admissible representations of any~$\RG_r$, in the sense of Jacquet,
Piatetski-Shapiro and Shalika (\cite{JPSS83}), which can also be defined 
through the Langlands-Shahidi method (\cite{Sh84}). 
Here~$\psi$ is a nontrivial additive character of~$F$;
the definition of this family of local gamma factors is
recalled in Section \ref{sec2}. It is natural to ask how 
this family of invariants yields
information about the representation~$\pi$.

In this paper, we consider the \emph{Local Converse Problem} for~$\RG_n$, 
which is to find the least integer~$n_0$
such that the family of local gamma factors~$\gamma(s,\pi\times\tau,\psi)$, 
with~$\tau$ running through all
irreducible generic representations of~$\RG_r$ for~$r=1,\ldots,n_0$,
determines the irreducible generic representation~$\pi$ of~$\RG_n$ up
to isomorphism. It is an easy consequence of the work of Jacquet,
Piatetski-Shapiro and Shalika (\cite{JPSS83}) that~$n_0\le n$. 
The work of Henniart~(\cite{H93}) shows that~$n_0\le n-1$, 
and the works of J.-P.~Chen~(\cite{Ch96} and~\cite{Ch06})
and Cogdell and Piatetski-Shapiro~(\cite{CPS99}) show that~$n_0\le n-2$, for~$n\ge 3$. 
A stronger statement (when~$n>4$)
is the following conjecture, which is usually attributed to H.~Jacquet. 

\begin{conj}[The Jacquet Conjecture on the Local Converse Problem]\label{conj1}
Let~$\pi_1$ and~$\pi_2$ be irreducible generic smooth representations
of~$\RG_n$. If their local gamma factors~$\gamma(s,\pi_1\times\tau,\psi)$
and~$\gamma(s,\pi_2\times\tau,\psi)$ are equal, as functions in the
complex variable~$s,$ for all irreducible generic
representations~$\tau$ of~$\RG_r,$ with~$r=1,\ldots,\left[\frac{n}{2}\right]$,
then~$\pi_1$ and~$\pi_2$ are equivalent as representations of~$\RG_n$.
\end{conj}

It is clear that the work (\cite{H93,Ch96,CPS99,Ch06}) confirms that 
for~$2\le n\le 4$, Conjecture~\ref{conj1} is a theorem. Indeed, 
for~$n=2$, Conjecture~\ref{conj1} was proved in 1970 by Jacquet and Langlands 
in their well-known book (\cite{JL70}), and for~$n=3$ it was proved in 1979 
by Jacquet, Piatetski-Shapiro and Shalika (\cite{JPSS79}). 
Following a standard argument, 
which was already known to the experts in the 1980s,
we deduce in Section~\ref{sec2.3} that Conjecture~\ref{conj1} is
equivalent to the following conjecture. 

\begin{conj}\label{conj3}
Assume that~$\pi_1$ and~$\pi_2$ are irreducible unitarizable supercuspidal
representations of~$\RG_{n}$. If their 
local gamma factors~$\gamma(s,\pi_1\times\tau,\psi)$
and~$\gamma(s,\pi_2\times\tau,\psi)$ are equal as functions in the
complex variable~$s$, for all irreducible supercuspidal representations~$\tau$ of~$\RG_r$
with~$r=1,\ldots,\left[\frac{n}{2}\right]$, then~$\pi_1$ and~$\pi_2$ are
equivalent as representations of~$\RG_n$.
\end{conj}

Since any irreducible supercuspidal representation of~$\RG_n$
has a nontrivial Whittaker model, it is natural to use this property, combined
with the local functional equation of the local Rankin--Selberg convolution
for~$\RG_n\times\RG_r$, to figure out a possible approach to prove
Conjecture~\ref{conj3}. This is in fact the idea behind the previous
attacks on the Local Converse Problem (\cite{JPSS79,H93,Ch96,Ch06}).
In this paper, we add a new idea to the argument in order to attempt
to reduce the twists 
down to~$r=1,\ldots,\left[\frac{n}{2}\right]$, i.e. Conjecture~\ref{conj3}. The
idea is to find Whittaker functions satisfying some special properties.

Let~$\RU_n$ be the unipotent radical of the standard Borel subgroup~$\RB_n$
of~$\RG_n$, which consists of all upper-triangular matrices.
Denote by~$\RP_n$ the mirabolic subgroup of~$\RG_n$, consisting of
matrices with last row equal to~$(0,\ldots,0,1)$. We also fix a
standard non-degenerate character~$\psi_n$ of~$\RU_n$ (see
Section~\ref{sec2.1}) so that all Whittaker functions are
implicitly~$\psi_n$-Whittaker functions.

\begin{defn}\label{def:sp}
Let~$\pi$ be an irreducible unitarizable supercuspidal representations
of~$\RG_n$ and let~$\bK$ be a compact-mod-centre open subgroup of~$\RG_n$.
A (non-zero) Whittaker function~$W_{\pi}$ for~$\pi$ is called
\emph{$\bK$-special} if it satisfies:
\[
W_{\pi_i}(g^{-1})=\overline{W_{\pi_i}(g)}\text{ for all }g\in \bK ,
\]
and~$\Supp W_{\pi}\subset \RU_n\bK~$, where~$\overline{\phantom{a}}$ denotes
complex conjugation.
\end{defn}

\begin{defn}\label{sw}
Let~$\pi_1$ and~$\pi_2$ be irreducible unitarizable supercuspidal
representations of~$\RG_n$ with the same central
character. Let~$W_{\pi_1}$ and~$W_{\pi_2}$ be 
(non-zero) Whittaker functions
for~$\pi_1$ and~$\pi_2$, respectively. We call~$(W_{\pi_1} ,
W_{\pi_2})$ a \emph{special pair} (of Whittaker functions)
for~$(\pi_1,\pi_2)$ if there exists a compact-mod-centre open
subgroup~$\bK$ of~$\RG_n $ such that~$W_{\pi_1}$ and~$W_{\pi_2}$ are
both~$\bK$-special and
\[
W_{\pi_1}(p)=W_{\pi_2}(p),\text{ for all }p\in \RP_n.
\]
\end{defn}

If a special pair of Whittaker functions as in Definition~\ref{sw} exists
for~$(\pi_1,\pi_2)$, we can prove that the representations $\pi_1$ and
$\pi_2$ are distinguished by their families of 
local gamma factors~$\gamma(s,\pi_i\times\tau,\psi)$, for~$\tau$ irreducible
supercuspidal representations of~$\RG_r,$
with~$r=1,\ldots,\left[\frac{n}{2}\right]$, by using a refinement of the argument
in~\cite{Ch96} and~\cite{Ch06}. This approach was successfully carried
out by the second-named author in~\cite{N12} for general linear groups over
finite fields. The key point is to find a refined decomposition
for~$\RG_n$ which reflects the symmetry carried in Definition~\ref{def:sp}.
We recall in Section~\ref{sec3.1} this refined decomposition. 

\begin{thm}\label{thm:sp}
Let~$\pi_1$ and~$\pi_2$ be irreducible unitarizable supercuspidal
representations of~$\RG_n$. Assume that a special
pair~$(W_{\pi_1},W_{\pi_2})$ exists for~$(\pi_1,\pi_2)$. 
If the local gamma factors~$\gamma(s,\pi_1\times\tau,\psi)$
and~$\gamma(s,\pi_2\times\tau,\psi)$ are equal as functions in the
complex variable~$s$, for all irreducible
supercuspidal representations~$\tau$ of~$\RG_r$
with~$r=1,\ldots,\left[\frac{n}{2}\right]$, then~$W_{\pi_1}= W_{\pi_2}$
and~$\pi_1$ and~$\pi_2$ are equivalent as representations of~$\RG_n$.
\end{thm}

In certain cases, one can prove the existence of special pairs for
irreducible unitarizable supercuspidal representations of~$\RG_n$ by
using the construction of supercuspidal representations in terms of
maximal simple types of Bushnell and Kutzko (\cite{BK93}) and the
explicit construction of Bessel functions of supercuspidal
representations due to Pa{\v{s}}k{\=u}nas and the third-named author (\cite{PS08}). Given an
irreducible supercuspidal representation~$\pi$ of~$\RG_n$, one of the
invariants associated to it, by Bushnell and Henniart in~\cite{BH96},
is its \emph{endo-class}~$\Theta(\pi)$. We prove:

\begin{prop}\label{prop:endo-equivalent}
Let~$\pi_1$,~$\pi_2$ be irreducible unitarizable supercuspidal representations
of~$\RG_n$ with the same endo-class. Then there is a special
pair~$(W_{\pi_1},W_{\pi_2})$ for~$(\pi_1,\pi_2)$.
\end{prop}

Theorem~\ref{thm:sp} with
Proposition~\ref{prop:endo-equivalent} implies, for example, that two
level zero irreducible unitarizable supercuspidal representations~$\pi_1,\pi_2$
of~$\RG_n$ can be distinguished by the set of 
local gamma factors~$\gamma(s,\pi_i\times\tau,\psi)$, for all irreducible
supercuspidal representations~$\tau$ of~$\RG_r$
with~$r=1,2,\ldots,\left[\frac{n}{2}\right]$. In fact, this is a special case of
a more general result, as follows.

Attached to an irreducible supercuspidal representation~$\pi$ of~$\RG_n$, via
its endo-class~$\Theta(\pi)$, is an invariant which we call
its~\emph{degree}~$\deg(\pi)$. The degree is an integer dividing~$n$:
for example,~$\deg(\pi)=1$ if and only if~$\pi$ is a twist of a level
zero representation; 
and if~$\deg(\pi)<n$ then~$\pi$ is invariant under a non-trivial
unramified character twist, though the converse is not true.
By using formulae on the conductors of pairs of
supercuspidal representations from~\cite{BHK98,BH03}, we 
immediately obtain the following corollary.

\begin{cor}\label{cor:nonmax}
Let~$\pi_1$ and~$\pi_2$ be irreducible unitarizable supercuspidal
representations of~$\RG_n$ and suppose that~$\deg(\pi_1)<n$. If the
local gamma factors~$\gamma(s,\pi_1\times\tau,\psi)$
and~$\gamma(s,\pi_2\times\tau,\psi)$ are equal as functions in the
complex variable~$s$, for all irreducible supercuspidal
representations~$\tau$ of~$\RG_r$ with~$r=1,\ldots,\left[\frac{n}{2}\right]$,
then~$\pi_1$ and~$\pi_2$ are equivalent as representations of~$\RG_n$.
\end{cor}

We end the paper with some discussion of the scope of the methods used
here, in particular of the obstacles to extending to the
case~$\deg(\pi_1)=n$ (see Remark~\ref{rmk:end}).

\section{Basics in the local Rankin--Selberg convolution}\label{sec2}

We start by recalling the basic facts about Whittaker models of
irreducible generic representations of~$\RG_n$ and
local gamma factors of Rankin--Selberg convolution type over the 
locally compact non-archimedean local field~$F$. 
We denote by~$\Fo_F$ the ring of integers in~$F$, by~$\Fp_F$ the prime
ideal in~$\Fo_F$, and by~$\Fk_F$ the residue field of~$F$, 
of cardinality~$q$ and characteristic~$p$; we also
write~$|\cdot|$ for the absolute value on~$F$, normalized to have
image~$q^{\BZ}$. 
We use analogous notation for extensions
of~$F$. We also fix, once and for all, an additive character~$\psi$
of~$F$ which is trivial on~$\Fp_F$ but nontrivial on~$\Fo_F$.

\subsection{Whittaker models}\label{sec2.1}

Let~$\RQ_n$ be the standard parabolic subgroup of~$\RG_n$ corresponding to the
partition~$(n-1,1).$ Then
\[
\RQ_n=\RZ_n\RP_n,
\]
where~$\RZ_n$ is the center of~$\RG_n$, and~$\RP_n$ is the mirabolic subgroup.

\begin{defn}
A character~$\psi_{\RU_n}$ of~$\RU_n$ is called \emph{non-degenerate} if its
normalizer in~$\RB_n$ is~$\RZ_n\RU_n$. We denote by~$\psi_n$ the
\emph{standard} non-degenerate character given by
\[
\psi_{n}(u)=\psi\left(\sum_{i=1}^{n-1}u_{i,i+1}\right),
\]
for~$u=(u_{i,j})\in \RU_n$.
\end{defn}

We call an irreducible smooth representation~$(\pi,V_\pi)$ of $\RG_n$ 
\emph{generic} if there is 
a non-degenerate character~$\psi_{\RU_n}$ of~$\RU_n$
such that the~$\Hom$-space
\[
\Hom_{\RG_n}(V_\pi,\Ind_{\RU_n}^{\RG_n}(\psi_{\RU_n}))
\]
is nonzero. By the uniqueness of local Whittaker models,
this~$\Hom$-space is at most one-dimensional. 
Since the non-degenerate characters of~$\RU_n$ are all conjugate
under~$\RB_n$, we see that~$\pi$ is generic if and only if
\[
\Hom_{\RG_n}(V_\pi,\Ind_{\RU_n}^{\RG_n}(\psi_n))
\cong \Hom_{\RU_n}(V_\pi|_{\RU_n}, \psi_n)
\]
is non-zero, where the isomorphism comes from Frobenius reciprocity.

Assume that~$\pi$ is generic. We fix a nonzero functional~$\ell$
in~$\Hom_{\RU_n}(V_\pi|_{\RU_n}, \psi_n)$ (which is unique up to scalar).
The \emph{Whittaker function} attached to a vector~$v\in V_\pi$ is defined by
\[
W_v(g):=\ell(\pi(g)(v)), \text{ for all } g\in \RG_n.
\]
It is easy to see that~$W_v$ belongs to~$\Ind_{\RU_n}^{\RG_n}(\psi_n)$ and
\[
\CW(\pi,\psi_n):=\{W_v \mid v\in V_{\pi}\}
\]
is called the~\emph{$\psi_n$-Whittaker model} of~$\pi,$ or simply the Whittaker
model of~$\pi$. It is clear that the Whittaker model of $\pi$ is 
independent of the choice of the nonzero functional~$\ell$.

For any~$W_v\in \CW(\pi, \psi_n)$, define
\[
\widetilde {W_v}(g):=W_v(w_n\cdot{^tg^{-1}}),
\]
for~$g\in \RG_n$, where~$w_n$ is the longest Weyl group element of~$\RG_n$,
with~$1$'s on the second diagonal and zeros elsewhere, and~${}^tg$ denotes the
transpose of~$g$. Then one can check that the function~$\widetilde
{W_v}$ belongs to the~$\psi_n^{-1}$-Whittaker model of the
contragredient~$\tilpi$ of~$\pi$, that is,
\[
\widetilde {W_v}\in \CW(\tilde \pi, \psi_n^{-1})\subset
\Ind_{\RU_n}^{\RG_n}(\psi_n^{-1}).
\]
It is a basic fact that any irreducible supercuspidal representation
of~$\RG_n$ is generic (\cite[Theorem~B]{GeK75}). We recall the following
properties of the restriction of an irreducible generic
representation~$(\pi,V_\pi)$ of~$\RG_n$ to the subgroup~$\RP_n$, which can be viewed as the 
starting point of our approach to prove the Jacquet conjecture for $\RG_n$. 

\begin{thm}[{\cite[\S5]{BeZ76}}]
With the notation fixed as above, the following hold.
\begin{enumerate}
\item $\Ind_{\RU_n}^{\RP_n}(\psi_n)$ is irreducible as a representation of~$\RP_n$.
\item If~$\pi$ is a generic representation of~$\RG_n$,
then~$\Ind_{\RU_n}^{\RP_n }(\psi_n)$ is a~$\RP_n$-sub\-representation
of~$\CW(\pi,\psi_n)|_{\RP_n}$.
\item If~$\pi$ is an irreducible supercuspidal representation of~$\RG_n,$
then~$\pi|_{\RP_n}$ is equivalent to~$\Ind_{\RU_n}^{\RP_n }(\psi_n)$ as
representations of~$\RP_n$.
\end{enumerate}
\end{thm}

\subsection{Local gamma factors}\label{sec2.2}

Next we review the basic setting of local gamma factors
attached to a pair of irreducible generic representations, for details
of which we refer to~\cite{JPSS83}.

Let~$n,r\ge 1$ be integers and let~$\pi$ and~$\tau$ be irreducible
generic representations of~$\RG_n$  and~$\RG_r$, respectively, with
central characters~$\omega_\pi$ and~$\omega_\tau$
respectively. Let~$W_\pi\in \CW(\pi,\psi_n)$ be a Whittaker function
of~$\pi$ and~$W_\tau\in \CW(\tau,\psi_r^{-1})$ be a Whittaker function
of~$\tau$. Since it is the only case of interest to us here, we
suppose that~$n>r$.

If~$j$ is an integer for which~$n-r-1\ge j\ge 0$, a local zeta integral for the
pair~$(\pi,\tau)$ is defined by
\[
\CZ(W_\pi, W_\tau,   s; j):=
\int_{g}\int_{x} W_\pi
\begin{pmatrix}g&0&0\\
x&\RI_j&0\\
0&0&\RI_{n-r-j}\end{pmatrix}
W_\tau(g)|\det g|^{s-\frac{n-r}{2}} dx dg,
\]
where the integration in the variable~$g$ is over~$\RU_r\backslash\RG_r$ and the
integration in the variable~$x$ is over $\Mat_{j\times r}(F)$.
Jacquet, Piatetski-Shapiro, and Shalika proved in~\cite{JPSS83} the
following theorem.

For~$g\in\RG_n$, we denote by~$R_g$ the right translation action
of~$g$ on
functions from~$\RG_n$ to~$\BC$, and we put~$w_{n,r}=
\begin{pmatrix}
\RI_{r}& 0\\
0&w_{n-r} \end{pmatrix}$.

\begin{thm}[{\cite[Section 2.7]{JPSS83}}]\label{FE}
With notation as above, the following hold.
\begin{enumerate}
\item Each integral~$\CZ(W_\pi, W_\tau, \Phi,  s; j)$ is absolutely convergent
for~$\Re(s)$ sufficiently large and is a rational function of~$q^{-s}$. More precisely,
for fixed~$j$, the integrals~$\CZ(W_\pi, W_\tau, s;j)$ span a fractional ideal
(independent of~$j$)
\[
\BC[q^s,q^{-s}]L(s,\pi\times\tau)
\]
of the ring~$\BC[q^s,q^{-s}]$, where the local~$L$-factor~$L(s, \pi\times\tau)$ has the
form~$P(q^s)^{-1}$, with~$P\in \BC[x]$ and~$P(0)=1$.
\item  For~$n-r-1\ge j\ge 0$, there is a factor~$\epsilon(s,\pi\times \tau, \psi)$
independent of~$j$, such that
\begin{eqnarray*}
&&\frac{\CZ(R_{w_{n,r}}\widetilde{W}_\pi,\widetilde{W}_\tau,  1-s; n-r-j-1)}
{L(1-s, \tilpi\times\tiltau)} \\
&&\quad\hskip4cm =\ \omega_\tau(-1)^{n-1}\epsilon( s, \pi\times \tau, \psi)
\frac{\CZ(W_\pi, W_\tau, s;j)}{L(s,\pi\times\tau)}.
\end{eqnarray*}
\item There are~$c\in\BC^\times$ and~$f=f(\pi\times\tau,\psi)\in\BZ$ such that
\[
\epsilon( s, \pi\times \tau, \psi)=cq^{-fs}.
\]
\end{enumerate}
\end{thm}

The local gamma factor
attached to a pair of representations~$\pi$ and~$\tau$ is defined
in~\cite{JPSS83} by
\begin{equation}\label{gammadef}
\gamma(s,\pi\times \tau, \psi)=\epsilon(s,\pi\times \tau, \psi)\frac{L(1-s,\tilpi\times\tiltau)}{L(s,\pi\times  \tau)}.
\end{equation}
Then the functional equation in Part~(ii) of Theorem~\ref{FE} can be rewritten
\begin{equation}\label{gam}
\CZ(R_{w_{n,r}}\widetilde{W}_\pi, \widetilde{W}_\tau, 1-s; n-r-j-1)
=\omega_\tau(-1)^{n-1}\gamma(s,\pi\times \tau, \psi)\CZ(W_\pi, W_\tau,  s,j).
\end{equation}
We also remark that the local gamma factor~$\gamma(s,\pi\times\tau,\psi)$
determines the conductor~$f(\pi\times\tau,\psi)$, since it is the leading power
of~$q^{-s}$ in a power series expansion for~$\gamma(s,\pi\times\tau,\psi)$.

\subsection{Central characters}

In this section, we show that the 
well-known result that local gamma factors determine the
central character. We begin by recalling the following result on the stability
of local gamma factors, which follows from~\cite[Proposition~2.7]{JS85}.

\begin{prop}\label{stab0}
Let~$\pi$ be an irreducible generic representation of~$\RG_n$ with $n\geq 2$. Then there
exits~$m_\pi$ such that, for any character~$\chi$ of~$F^\times$ of
conductor~$m\ge m_{\pi}$ and any~$c\in \Fp^{-m}$ satisfying~$\chi(1+x)=\psi(cx)$,
for~$x\in\Fp_F^{\left[\frac{m}{2}\right]+1}$, we have
\[
L(s,\pi\times\chi)=1
\text{ and }
\epsilon(s, \pi\times \chi, \psi)=\omega_{\pi}(c)^{-1} \epsilon(s, 1\times\chi, \psi)^n.
\]
\end{prop}

\begin{proof}
Although this is not quite the statement
of~\cite[Proposition~2.7]{JS85}, this statement is included in the
proof (see page 323 of \emph{op.\ cit.}).
\end{proof}

\begin{cor}\label{char0} Let~$\pi_1$,~$\pi_2$ be irreducible generic
representations of~$\RG_n$. If their local gamma factors~$\gamma(s,\pi_1\times\chi,\psi)$
and~$\gamma(s,\pi_2\times\chi,\psi)$ are equal as functions in the
complex variable~$s,$ for any character~$\chi$ of~$F^\times$,
then~$\omega_{\pi_1}=\omega_{\pi_2}.$
\end{cor}

\begin{proof}
For~$i=1,2$, let~$m_{\pi_i},m_{\tilpi_i}$ be the numbers given by 
Proposition~\ref{stab0}
and put~$m_0=\max\{m_{\pi_i},m_{\tilpi_i}\mid i=1,2\}$. For~$\chi$ a character
of~$F^\times$ of conductor~$m\ge m_0$, we
have~$\epsilon(s,\pi_i\times\chi,\psi)=\gamma(s,\pi_i\times\chi,\psi)$,
by~\eqref{gammadef} and 
Proposition~\ref{stab0}.

For any~$c\in\Fp^{-m}\setminus\Fp^{1-m}$, with~$m\ge m_0$,
there exists a character~$\chi_c$ character of conductor~$m$ such
that~$\chi_c(1+x)=\psi(cx)$, for~$x\in \Fp_F^{\left[\frac{m}{2}\right]+1}$; thus
Proposition~\ref{stab0} implies
\[
\omega_{\pi_1}(c)=\omega_{\pi_2}(c).
\]
Since any element of~$F^\times$ can be expressed as the quotient of
two elements of valuation at most~$-m$, we deduce
that~$\omega_{\pi_1}=\omega_{\pi_2}.$
\end{proof}

\subsection{Reduction from generic to supercuspidal}\label{sec2.3}

This section is devoted to reducing Conjecture~\ref{conj1} to
Conjecture~\ref{conj3}. In other words, 
if the Local Converse Theorem for twisting by generic representations
of rank up to~$\left[\frac{n}{2}\right]$ holds for unitarizable supercuspidal
representations, then it also holds for general generic smooth representations.

Let~$\pi$ be an irreducible generic smooth representation of~$\RG_n$. 
From the classification of irreducible smooth representations
of~$\RG_n$~%
\cite[Theorem~9.7]{Z80},~$\pi$ is the unique irreducible
generic subquotient of a standard parabolically induced representation
\[
\tau_1|\cdot|^{z_1}\times\cdots\times\tau_t|\cdot|^{z_t},
\]
where each~$\tau_i$ is an irreducible unitarizable supercuspidal representation
of~$\RG_{n_i}$, with~$n=\sum_{i=1}^t n_i$, and
\[
z_1\ge \cdots\ge z_t
\]
are real numbers. Moreover~$(\tau_1,\ldots,\tau_t)$ and~$(z_1,\ldots,z_t)$ are uniquely
determined up to a permutation~$\sigma$ such that~$z_{\sigma(i)}=z_i$, and any such
tuples give rise to an irreducible generic representation of~$\RG_n$ in this way. By
the multiplicativity of the 
local gamma factors (\cite[Theorem~3.1]{JPSS83}),we have
\begin{equation}\label{eqn:mult}
\gamma(s,\pi\times\tau,\psi) = \prod_{i=1}^t\gamma(s+z_i,\tau_i\times\tau,\psi),
\end{equation}
for all irreducible generic representations~$\tau$ of~$\RG_r$.
We also observe that there is at most one index~$i$ such that~$n_i>\left[\frac n2\right]$.

\begin{prop}[{\cite[Section~3.2]{JiS03}}]\label{prop:poles}
With notation as above, assume that $\tau$ is irreducible, unitarizable and supercuspidal. 
\begin{enumerate}
\item If~$\prod_{i=1}^t\gamma(s+z_i,\tau_i\times\tau,\psi)$ has a real pole
(respectively, zero) at~$s=s_0$, then~$\tau\simeq\tiltau_i$ and~$s_0=1-z_i$
(respectively,~$s_0=-z_i$), for some~$i\in\{1,\ldots, t\}$. 
\item For each~$j=1,\ldots, t$, the
product~$\prod_{i=1}^t\gamma(s+z_i,\tau_i\times\tiltau_j,\psi)$ has a real pole and zero. 
Moreover, if~$j=1$ then there is a zero at~$s=-z_1$, and if~$j=t$ then there is a
pole at~$s=1-z_t$.
\end{enumerate}
\end{prop}

Note that the assumption in~\cite{JiS03}, that~$F$ is of
characteristic zero, is not used in the proof of this since it
requires only the multiplicativity of local gamma factors.

\begin{cor}\label{cor:multinv}
With notation as above, suppose also that~$\tau'_i$ are irreducible unitarizable
supercuspidal representation of~$\RG_{n'_i}$, for~$1\le i\le t'$,
with~$n=\sum_{i=1}^{t'} n'_i$, and that~$z'_1\ge\cdots\ge z'_{t'}$ are real numbers.
Suppose~$m\ge\left[\frac n2\right]$ and
\[
\prod_{i=1}^t\gamma(s+z_i,\tau_i\times\tau,\psi) =
\prod_{i=1}^{t'}\gamma(s+z'_i,\tau'_i\times\tau,\psi),
\]
for all irreducible unitarizable supercuspidal representations~$\tau$ of~$\RG_r$,
with~$r=1,2,\ldots,m$. Then~$t=t'$ and there is a permutation~$\sigma$
of~$\{1,\ldots t\}$ such that:
\begin{enumerate}
\item $n_i=n'_{\sigma(i)}$, for all~$i=1,\ldots,t$;
\item $\gamma(s+z_i,\tau_i\times\tau,\psi)=
\gamma(s+z'_{\sigma(i)},\tau'_{\sigma(i)}\times\tau,\psi)$,
for all irreducible unitarizable supercuspidal representations~$\tau$ of~$\RG_r$,
with~$r=1,2,\ldots,m$ and~$i=1,\ldots,t$;
\item $\tau_i\simeq\tau'_{\sigma(i)}$ and~$z_i=z'_{\sigma(i)}$, for all~$i$ such
that~$n_i\le\left[\frac n2\right]$.
\end{enumerate}
\end{cor}

\begin{proof} The proof is by induction on~$t$. If~$t=1$ but~$t'>1$
then~$n'_j\le\left[\frac n2\right]$, for some~$j$, and, by
Proposition~\ref{prop:poles}(i),~$\prod_{i=1}^{t'}\gamma(s+z'_i,\tau'_i\times\tiltau'_j,\psi)$
has a real pole while~$\gamma(s+z_1,\tau_1\times\tiltau'_j,\psi)$ does not, which is
absurd. Thus~$t'=1$ and there is nothing more to prove.

Now assume~$t\ge 2$ and note that either~$n_1$ or~$n_t$ is at most~$\left[\frac n2\right]$.
Suppose first that~$n_t\le\left[\frac n2\right]$.
Then~$\prod_{i=1}^t\gamma(s+z_i,\tau_i\times\tiltau_t,\psi)$ has a pole at~$1-z_t$ so,
by~Proposition~\ref{prop:poles}(i), there is an integer~$1\le j\le t'$ such
that~$\tau'_j\simeq\tau_t$ and~$z'_j=z_t$.
Hence~$\tau'_j|\cdot|^{z'_j}\simeq\tau_t|\cdot|^{z_t}$
and~$\gamma(s+z_t,\tau_t\times\tau,\psi) = \gamma(s+z'_j,\tau'_j\times\tau,\psi)$,
for all irreducible generic representations~$\tau$ of~$G_r$, for all~$r$. In particular,
we deduce
\[
\prod_{i=1}^{t-1}\gamma(s+z_i,\tau_i\times\tau,\psi) =
\prod_{i=1,i\ne j}^{t'}\gamma(s+z'_i,\tau'_i\times\tau,\psi),
\]
for all irreducible unitarizable supercuspidal representations~$\tau$ of~$\RG_r$,
with~$r=1,2,\ldots,m$. The result now follows from the inductive hypothesis.

Finally, if~$n_1\le\left[\frac n2\right]$ then~$\prod_{i=1}^t\gamma(s+z_i,\tau_i\times\tiltau_1,\psi)$
has a zero at~$-z_1$ so, by~Proposition~\ref{prop:poles}(i), there is an
integer~$1\le j\le t'$ such that~$\tau'_j\simeq\tau_1$ and~$z'_j=z_1$. The result then
follows as in the first case.
\end{proof}

Putting Corollary~\ref{cor:multinv} with~$m=\left[\frac n2\right]$ together with the
multiplicativity of local gamma factors~\eqref{eqn:mult} and the classification
of irreducible generic representations~\cite[Theorem~9.7]{Z80}, 
we see that Conjecture~\ref{conj3} implies Conjecture~\ref{conj1}.

\section{Special pairs and the local converse theorem}

\subsection{Preliminary results}\label{sec3.1}

We begin by recalling some useful lemmas, which form the technical
steps of the proof.

\begin{lem}[{\cite[Section~3.2]{JS85}}]\label{Jac}
Let~$t$ be a positive integer and let~$H$ be a complex smooth function
on~$\RG_t$  with compact support modulo~$\RU_t$ satisfying 
\[
H(ug)=\psi_t(u)H(g),
\]
for all~$u\in \RU_t, g\in\RG_t$.
If
\[
\int_{\RU_t\backslash\RG_t}H(g)W_\tau(g)dg=0,
\]
for all~$W_\tau\in\CW(\tau, \psi_t^{-1}),$
with~$\tau$ running through
all irreducible generic representations of~$\RG_t,$
then~$H\equiv 0.$
\end{lem}

From~\cite[Section 3.1]{Ch06}, we have the generalized Bruhat decomposition:
\[
\RG_n=\bigsqcup_{i=0}^{n-1}\RU_n\alpha^i\RQ_n,
\]
where~$\alpha=\begin{pmatrix}
0&\RI_{n-1}\\
1&0
\end{pmatrix}.$

\begin{defn}
Given two functions~$H_1$ and~$H_2$ on~$\RG_n,$ if
\[
H_1(x)=H_2(x),
\text{ for all } x\in \RU_n\alpha^i\RQ_n,
\]
then we say that~$H_1$ and~$H_2$ \emph{agree on height~$i$}.
\end{defn}

Assume that~$(W_{\pi_1},W_{\pi_2})$ is a special pair for~$(\pi_1,\pi_2)$, as in
Definition~\ref{sw}, so that~$W_{\pi_1}$ and~$W_{\pi_2}$ agree on
height~$i=0$. The condition on local gamma factors
in the statement Conjecture~\ref{conj3}, via 
Corollary~\ref{char0} and
the following proposition, implies the agreement 
of~$W_{\pi_1}$ and~$W_{\pi_2}$ on height $i,$ for $i=0, \ldots , \left[\frac{n}{2}\right]$.

\begin{prop}[{\cite[Proposition 3.1]{Ch06}}]\label{htr}
Fix an integer~$1\le r< n$.
Let~$\pi_1$ and~$\pi_2$ be irreducible supercuspidal representations
of~$\RG_n$ with the same central character, 
and let~$W_{\pi_1},W_{\pi_2}$ be Whittaker functions, for~$\pi_1,\pi_2$ 
respectively, which coincide on~$\RP_n$. If the 
local gamma factors~$\gamma(s,\pi_1\times\tau,\psi)$
and~$\gamma(s,\pi_2\times\tau,\psi)$ are equal as functions in the
complex variable~$s\in\BC$, for all irreducible generic representations~$\tau$ of~$\RG_r$,
then the two Whittaker functions~$W_{\pi_1}$,~$W_{\pi_2}$ agree on height~$r$.
\end{prop}

We are going to use the functional equations together with the properties of
special pairs of Whittaker functions in order to show that if a special
pair~$(W_{\pi_1},W_{\pi_2})$ agree on height $i$, for $i=0, \ldots, \left[\frac n2\right]$, then they are in fact
equal. To do so, we apply a refined decomposition of~$\RG_n$, whose finite field
version was a key ingredient in the proof of the Jacquet Conjecture on the Local Converse Problem for~$\RG_n$
over finite fields in~\cite{N12}.

\begin{prop}[{\cite[Proposition~3.8]{N12}}]
The following (non-disjoint) decomposition holds:
\[
\RG_n= \bigcup_{0\le r\le \left[\frac{n}{2}\right], n-\left[\frac{n}{2}\right]\le k\le n}\RU_n\alpha^r\RQ_n\alpha^k\RU_n.
\]
\end{prop}

\subsection{Proof of Theorem~\ref{thm:sp}}

Let~$\pi_1$,~$\pi_2$ be irreducible supercuspidal representations of~$\RG_n$ and
let~$(W_{\pi_1},W_{\pi_2})$ be a special pair for~$(\pi_1,\pi_2)$. Let~$\bK$ be the
compact-mod-centre open subgroup of~$\RG_n$ such that~$W_{\pi_i}$ are~$\bK$-special.
By hypothesis, the local gamma factors~$\gamma(s,\pi_1\times\tau,\psi)$
and~$\gamma(s,\pi_2\times\tau,\psi)$ are equal as functions in the complex
variable~$s$, for all irreducible supercuspidal representations~$\tau$
of~$\RG_r$ with~$r=1,\ldots,\left[\frac{n}{2}\right]$. This condition can be extended for 
all irreducible generic smooth representations $\tau$ of $\RG_n$ by
the multiplicativity of local gamma factors. 
Moreover, by Corollary~\ref{char0},~$\pi_1,\pi_2$ have the same
central character.
The proof goes in three steps.

{\bf Step~(1).}\
By Proposition~\ref{htr},~$W_{\pi_1}(g)=W_{\pi_2}(g),$ for
\[
g\in \bigcup_{0\le r\le \left[\frac{n}{2}\right]}\RU_n\alpha^r\RQ_n =
\bigcup_{0\le r\le \left[\frac{n}{2}\right]}\RU_n\alpha^r\RQ_n\alpha^n\RU_n.
\]

{\bf Step~(2).}\
For~$g=q\alpha^ku\in\RQ_n\alpha^k\RU_n\cap \bK$, 
with~$n-\left[\frac n2\right]\le k\le n$, and~$i=1,2$, we have
\[
W_{\pi_i}(q\alpha^ku)
=
\overline{W_{\pi_i}((q\alpha^ku)^{-1})}
=
\overline{W_{\pi_i}(u^{-1}\alpha^{n-k}q^{-1})},
\]
since~$W_{\pi_i}$ is~$\bK$-special.
Since~$u^{-1}\alpha^{n-k}q^{-1}\in \RU_n\alpha^{n-k}\RQ_n$, from {\bf Step~(1)} it follows that
\[
W_{\pi_1}(q\alpha^ku)=W_{\pi_2}(q\alpha^ku).
\]
Thus~$W_{\pi_1}$,~$W_{\pi_2}$ agree on~$\RQ_n\alpha^k\RU_n\cap \bK$ and hence
on~$\RQ_n\alpha^k\RU_n\cap \RU_n\bK$, since they are both~$\psi_n$-Whittaker functions.
Since~$\Supp W_{\pi_i}\subset \RU_n\bK,$ we deduce that~$W_{\pi_1}(g)=W_{\pi_2}(g)$,
for all
\[
g\in \bigcup_{n-\left[\frac n2\right]\le k\le n}\RQ_n\alpha^k\RU_n
=\bigcup_{n-\left[\frac n2\right]\le k\le n}\RU_n\alpha^0\RQ_n\alpha^k\RU_n.
\]

{\bf Step~(3).}\
It remains to consider the case of~$g\in\RU_n\alpha^r\RQ_n\alpha^k\RU_n$,
with~$1\le r\le \left[\frac{n}{2}\right]$ and~$n-\left[\frac{n}{2}\right]\le k\le n-1$.
For any fixed~$u\in\RU_n$ and~$p\in \RP_n$, {\bf Step~(2)} implies that
\[
 R_{p\alpha^{k}u}W_{\pi_1}(q)  = R_{p\alpha^{k}u}W_{\pi_2}(q),
\]
for all~$q\in \RP_n$, where we recall that~$R_g$ denotes the right translation action
by~$g$ on the Whittaker functions. We apply the functional equation~\eqref{gam}
for~$j=n-r-1$ to the Whittaker functions~$R_{p\alpha^{k}u}W_{\pi_i}$ for~$i=1,2$ and
any Whittaker function~$W_\tau$ in~$\CW(\tau,\psi_r^{-1})$.

The local zeta function~$\CZ(R_{p\alpha^{k}u}W_{\pi_i},W_\tau,s;n-r-1)$ is given by
the following integral
\[
\int_{h}
\int_{x} R_{p\alpha^{k}u}W_{\pi_i}
 \begin{pmatrix}h&0&0\\
x&\RI_{n-r-1}&0\\
0&0&1\end{pmatrix} W_\tau(h)|\det h|^{ s-{ \frac{n-r}{2}}}dx dh
\]
where the integration in the variable~$h$ is over~$\RU_r\bks\RG_r$ and the integration
in the variable~$x$ is over~$\Mat_{(n-r-1)\times r}(F)$. Hence we obtain
\[
\CZ(R_{p\alpha^{k}u}W_{\pi_1},W_\tau,s;n-r-1)
=
\CZ(R_{p\alpha^{k}u}W_{\pi_2},W_\tau,s;n-r-1).
\]
Since~$\gamma(s,\pi_1\times\tau,\psi)=\gamma(s,\pi_2\times\tau,\psi)$, 
by the functional equation~\eqref{gam} for~$j=n-r-1$, we obtain
\[
\CZ(R_{w_{n,r}}\widetilde{R_{p\alpha^{k}u}W_{\pi_1}},\widetilde{W}_\tau,1-s;0)
=
\CZ(R_{w_{n,r}}\widetilde{R_{p\alpha^{k}u}W_{\pi_2}},\widetilde{W}_\tau,1-s;0).
\]
Thus, from the definition of these zeta integrals,
\[
\int_{g} \left(R_{w_{n,r}}\widetilde{R_{p\alpha^{k}u}W_{\pi_1}}-
R_{w_{n,r}}\widetilde{R_{p\alpha^{k}u}W_{\pi_2}}\right)
\begin{pmatrix}g&0\\
0&\RI_{n-r}\end{pmatrix}
|\det(g)|^{s-\frac{n-r}2} \widetilde{W}_\tau(g) dg = 0,
\]
for all generic representations~$\tau$ of~$\RG_r$, where the integration in
the variable~$h$ is over~$\RU_r\bks\RG_r$. From Lemma~\ref{Jac}, we deduce that
\[
R_{w_{n,r}}\widetilde{R_{p\alpha^{k}u}W_{\pi_1}}\begin{pmatrix}g&0\\0&\RI_{n-r}\end{pmatrix}
=R_{w_{n,r}}\widetilde{R_{p\alpha^{k}u}W_{\pi_2}}\begin{pmatrix}g&0\\0&\RI_{n-r}\end{pmatrix}
\]
for all~$p\in\RP_n$,~$u\in\RU_n$ and~$g\in\RG_r$. Now by definition, for~$i=1,2$,
\begin{eqnarray*}
R_{w_{n,r}}\widetilde{R_{p\alpha^{k}u}W_{\pi_i}}\begin{pmatrix}g&0\\0&\RI_{n-r}\end{pmatrix}
&=&
R_{p\alpha^{k}u}W_{\pi_i}\left(w_n\begin{pmatrix}{^tg^{-1}}&0\\0&I_{n-r}\end{pmatrix}{^tw_{n,r}^{-1}}\right)\\
&=&
W_{\pi_i}\left(\begin{pmatrix}0&\RI_{n-r}\\w_r{^tg^{-1}}&0\end{pmatrix}p\alpha^ku\right).
\end{eqnarray*}
Hence we obtain the identity
\[
W_{\pi_1}\left(\begin{pmatrix}0&\RI_{n-r}\\w_r{^tg^{-1}}&0\end{pmatrix}p\alpha^ku\right)
=
W_{\pi_2}\left(\begin{pmatrix}0&\RI_{n-r}\\w_r{^tg^{-1}}&0\end{pmatrix}p\alpha^ku\right),
\]
for all~$p\in\RP_n$,~$u\in\RU_n$ and~$g\in\RG_r$. In particular, taking~$g=w_r$ we obtain
\[
W_{\pi_1}(\alpha^rp\alpha^ku) = W_{\pi_2}(\alpha^rp\alpha^ku),
\]
for all~$p\in\RP_n$ and~$u\in\RU_n$. This proves that~$W_{\pi_1}(g)=W_{\pi_2}(g)$, for
\[
g\in\RU_n\alpha^r\RQ_n\alpha^k\RU_n
\]
with~$1\le r\le \left[\frac{n}{2}\right]$ and~$n-\left[\frac{n}{2}\right]\le k\le n-1$. This completes {\bf Step~(3)}.

By combining the results from all three {\bf Steps} above, we obtain that
\[
W_{\pi_1}(g)=W_{\pi_2}(g),\text{ for all } g\in\RG_n.
\]
By the uniqueness of local Whittaker models for irreducible smooth representations
of~$\RG_n$, the two Whittaker models~$\CW(\pi_1,\psi_n)$ and~$\CW(\pi_2,\psi_n)$ have
trivial intersection unless~$\pi_1$ and~$\pi_2$ are equivalent
as representations of~$\RG_n$, which completes the proof of Theorem~\ref{thm:sp}.

\section{Supercuspidals with the same endo-class}

$\bK$-special Whittaker functions are Whittaker functions
of~$\RG_n$ with certain symmetry when restricted to~$\bK$. The Bessel functions
of irreducible supercuspidal representations of~$\RG_n$ constructed by Pa{\v{s}}k{\=u}nas and the third-named author in~\cite{PS08} 
are such examples. We recall from~\cite{PS08} the
basics of these Bessel functions, which rely on the construction theory of supercuspidal representations 
of~$\RG_n$ in terms of maximal simple types of Bushnell and Kutzko
\cite{BK93}. We will use the standard notation from~\cite{BK93} and~\cite{PS08}.

\subsection{Bessel functions}

We begin by recalling from~\cite[Section 5]{PS08} the general formulation of
Bessel functions. Let~$\CK$ be an open compact-modulo-center subgroup of~$\RG_n$
and let~$\CU\subset\CM\subset\CK$ be compact open subgroups of~$\CK$.
Let~$\tau$ be an irreducible smooth representation of~$\CK$ and let~$\Psi$ be a
linear character of~$\CU$. Take an open normal subgroup~$\CN$ of~$\CK$, which is
contained in~$\Ker(\tau)\cap\CU$. 
Let~$\chi_\tau$ be the (trace) character of~$\tau$. The associated
\emph{Bessel function}~$\CJ:\CK\rightarrow\BC$ of~$\tau$ is defined by 
\[
\CJ(g):=[\CU:\CN]^{-1}\sum_{u\in\CU/\CN}\Psi(h^{-1})\chi_\tau(gu).
\]
This is independent of the choice of~$\CN$. The basic properties of this Bessel
function which we will need are given below.

\begin{prop}[{\cite[Proposition~5.3]{PS08}}]\label{bs30}
Assume that the data introduced above satisfy the following:
\begin{itemize}
\item $\tau|_\CM$ is an irreducible representation of~$\CM$; and
\item $\tau|_\CM\cong \Ind^\CM_\CU(\Psi)$.
\end{itemize}
Then the Bessel function~$\CJ$ of~$\tau$ enjoys the following properties:
\begin{enumerate}
\item $\CJ(1)=1$;
\item $\CJ(hg)=\CJ(gh)=\Psi(h)\CJ(g)$ for all~$h\in\CU$ and~$g\in\CK$;
\item if~$\CJ(g)\neq 0$, then~$g$ intertwines~$\Psi$; in particular, if~$m\in\CM$, then~$\CJ(m)\neq0$ if and only if~$m\in\CU$;
\end{enumerate}
\end{prop}

When the representation~$\tau$ is also unitarizable, the Bessel function enjoys
another symmetry property, as in the finite field case in~\cite{N12}.

\begin{lem}\label{inv}
In the situation of Proposition~\ref{bs30}, assume further that~$\tau$
is unitarizable. Then
\[
\CJ(g)=\overline{\CJ(g^{-1})}, \quad\text{ for }g\in\CK.
\]
\end{lem}

\begin{proof}
Note that~$\Psi$ is unitary, since it is a character of the compact group~$\CU.$ That
is~$\overline{\Psi(g^{-1})}=\Psi(g).$ Since~$\chi_{\tau}$ is also unitary
and~$\chi_{\tau}(gh)=\chi_{\tau}(hg)$, for~$g,h \in \CK$, we get
\begin{eqnarray*}
\overline{\CJ (g^{-1})}&:=&
[\CU:\CN]^{-1}\overline{\sum _{u\in\CU/  \CN}\Psi(u ^{-1})\chi_{\tau}(g^{-1}u )} \\
&=&[\CU:\CN]^{-1}\sum_{u \in\CU/  \CN}\Psi(u )\chi_{\tau}( u^{-1}g )\\
&=&[\CU:\CN]^{-1}\sum_{u \in\CU/  \CN}\Psi(u )\chi_{\tau}(gu^{-1})\\
&=&[\CU:\CN]^{-1}\sum_{u \in\CU/  \CN}\Psi(u^{-1}  )\chi_{\tau}( gu )\\
&=& \CJ (g),\text{ for }g\in \CK.
\end{eqnarray*}
The penultimate equality follows from the substitution~$u \mapsto u^{-1}$
and the normality of~$\CN$ in~$\CU.$
\end{proof}

\subsection{Maximal simple types}

Following~\cite[Section 6]{BK93}, the irreducible supercuspidal representations
of~$\RG_n$ are classified by means of maximal simple types~$(J,\lambda)$, where~$J$
is a compact open subgroup of~$\RG_n$ and~$\lambda$ is an irreducible representation
of~$J$. More precisely,~$(J,\lambda)$ is introduced as follows. We refer to~\cite{BK93}
for precise definitions of the objects introduced here.

Let~$V=F^n$, an~$n$-dimensional vector space over~$F$ with standard basis.
Thus we identify~$\Aut_F(V)$ with~$\RG_n$ and~$A=\End_F(V)$ with~$\Mat_{n\times n}(F)$.
Let~$\FA$ be a principal hereditary~$\Fo_F$-order in~$A$ with
Jacobson radical~$\FP$. Define~$\bU^0(\FA)=\bU(\FA)=\FA^\times$ and for~$m\ge 1$,
define~$\bU^m(\FA)=1+\FP^m$. For~$m\ge 0$, choose~$\beta\in A$ such
that~$\beta\in\FP^{-m}\setminus\FP^{1-m}$,~$E=F[\beta]$ is a field extension of~$F$,
and~$E^\times$ normalizes~$\FA$. Provided an additional technical condition is satisfied
(namely~$k_F(\beta)<0$), these data give a principal simple stratum~$[\FA,m,0,\beta]$
of~$A$. Take~$J=J(\beta,\FA)$,~$J^1=J^1(\beta,\FA)$, and~$H^1=H^1(\beta,\FA)$ as defined
in~\cite[Section~3]{BK93}. Denote by~$\CC(\FA,\beta,\psi)$ the set of simple (linear)
characters of~$H^1$ as defined in~\cite[Section~3]{BK93}.

Recall from~\cite[Section 6]{BK93} the following definition of maximal simple types.

\begin{defn}\label{mst}
The pair~$(J,\lambda)$ is called a \emph{maximal simple type} if one of the following holds:
\begin{itemize}
\item[(a)] $J=J(\beta,\FA)$ is an open compact subgroup associated to a simple
stratum~$[\FA,m,0,\beta]$ of~$A$ as above, such that, if~$E=F[\beta]$ and~$B=\End_E(V)$,
then~$\FB=\FA\cap B$ is a maximal~$\Fo_E$-order in~$B$. Moreover, there exists a
simple character~$\theta\in\CC(\FA,\beta,\psi)$ such that
\[
\lambda\cong\kappa\otimes\sigma,
\]
where~$\kappa$ is a~$\beta$-extension of the unique irreducible
representation~$\eta$ of~$J^1=J^1(\beta,\FA)$, which contains~$\theta$,
and~$\sigma$ is the inflation to~$J$ of an irreducible cuspidal representation of
\[
J/J^1\cong\bU(\FB)/\bU^1(\FB)\cong\RG\RL_r(\Fk_E),
\]
where~$r=n/[E:F]$.
\item[(b)] $(J,\lambda)=(\bU(\FA),\sigma)$, where~$\FA$ is a maximal
hereditary~$\Fo_F$-order in~$A$ and~$\sigma$ is the inflation to~$\bU(\FA)$ of an
irreducible cuspidal representation of
\[
\bU(\FA)/\bU^1(\FA)\cong\RG\RL_n(\Fk_F).
\]
\end{itemize}
\end{defn}
We will regard case~(b) formally as a special case of case~(a) by setting~$\beta=0$
and~$E=F$, and~$\theta,\eta,\kappa$ trivial.
In either case, we put~$\bJ=E^\times J$.
With these data, any irreducible supercuspidal representation~$\pi$ of~$\RG_n$ is of the form
\[
\pi\cong \cInd^{\RG_n}_{\bJ}(\Lambda),
\]
for some choice of~$(\bJ,\Lambda)$, where~$\Lambda|_J=\lambda$. We call such a
pair~$(\bJ,\Lambda)$ and \emph{extended maximal simple type}.

For~$\pi$ an irreducible supercuspidal representation of~$\RG_n$, any two extended
maximal simple types in~$\pi$ are conjugate in~$\RG_n$. This fact allows one to
associate some invariants to~$\pi$. The simple character~$\theta$ in the construction
of an extended maximal simple type for~$\pi$ determines an
\emph{endo-class}~$\Theta=\Theta(\pi)$ as defined in~\cite{BH96}. We do not recall
precisely the definition of endo-class: it is a class for a certain equivalence relation
on functions which take values in simple characters. For~$i=1,2$,
let~$\theta_i\in\CC(\FA_i,\beta_i,\psi)$ be simple characters for~$\RG_n$.
If~$\theta_1,\theta_2$ have the same endo-class then they intertwine in~$\RG_n$;
if, moreover, the hereditary orders~$\FA_1,\FA_2$ are isomorphic then~$\theta_1,\theta_2$
are conjugate in~$\RG_n$.

Although the field extension~$E/F$ involved in the construction of a maximal simple type
in~$\pi$ is not uniquely determined, its residue degree and ramification index are in fact
invariants of the endo-class~$\Theta=\Theta(\pi)$ and we write
\[
f(\Theta)=f(E/F),~e(\Theta)=e(E/F),~\deg(\Theta)=[E:F].
\]
These are then also invariants of~$\pi$ so we write~$\deg(\pi)=\deg(\Theta)$ and call
it the \emph{degree} of~$\pi$. We also remark that the~$\Fo_F$-period of the hereditary
order~$\FA$ in the construction of any maximal simple type in~$\pi$ is~$e(\Theta)$.

\subsection{Explicit Whittaker functions}

Let~$\pi$ be an irreducible \emph{unitarizable} supercuspidal representation of~$\RG_n$.
By~\cite[Proposition 1.6]{BH98}, there is an extended maximal simple
type~$(\bJ,\Lambda)$ in~$\pi$ such that
\[
\Hom_{\RU_n\cap\bJ}(\psi_n,\Lambda)\neq0.
\]
Since~$\Lambda$ restricts to a multiple of some simple
character~$\theta\in\CC(\FA,\beta,\psi)$, one obtains that~$\theta(u)=\psi_n(u)$ for
all~$u\in\RU_n\cap H^1$. As in~\cite[Definition 4.2]{PS08}, one defines a
character~$\Psi_n:(J\cap\RU_n)H^1\rightarrow\BC^\times$ by
\begin{equation}\label{Psin}
\Psi_n(uh):=\psi_n(u)\theta(h),
\end{equation}
for all~$u\in J\cap\RU_n$ and~$h\in H^1$. By~\cite[Theorem 4.4]{PS08}, the data
\[
\CK=\bJ,~\tau=\Lambda,~\CM=(J\cap\RP_n)J^1,~\CU=(J\cap\RU_n)H^1,~\text{and}~\Psi=\Psi_n
\]
satisfy the conditions in Proposition~\ref{bs30} and hence define a Bessel function~$\CJ$.

Now we define a function~$W_\pi:\RG_n\to\BC$ by
\begin{equation}\label{eqn:Wdef}
W_\pi(g):=\begin{cases}
\psi_n(u)\CJ(j)&\text{ if }g=uj\text{ with }u\in\RU_n,~j\in\bJ,\\
0&\text{ otherwise},
\end{cases}
\end{equation}
which is well-defined by Proposition~\ref{bs30}(ii).
Then, by~\cite[Theorem 5.8]{PS08},~$W_\pi$ is a Whittaker function for~$\pi$.
Moreover, since~$\pi$ is unitarizable, the same is true of~$\Lambda$, so~$W_\pi$ is
a~$\bJ$-special Whittaker function for~$\pi$, by Lemma~\ref{inv}. By
Proposition~\ref{bs30}, the restriction of~$W_\pi$ to~$\RP_n$ has a particularly
simple description: for~$g\in\RP_n$,
\begin{equation}\label{eqn:WsponPn}
W_\pi(g)=\begin{cases}
\Psi_n(g)&\text{ if }g\in(J\cap\RU_n)H^1;\\
0&\text{ otherwise}.
\end{cases}
\end{equation}

\subsection{Proof of Proposition~\ref{prop:endo-equivalent}}

Let~$\pi_1$,~$\pi_2$ be irreducible unitarizable supercuspidal representations
of~$\RG_n$ with the same endo-class~$\Theta$. We will use all
the notation of Definition~\ref{mst} but with subscripts~${}_1,{}_2$.

Let~$(\bJ_1,\Lambda_1)$ be an extended maximal
simple type in~$\pi_1$ such that
\[
\Hom_{\RU_n\cap\bJ_1}(\psi_n,\Lambda_1)\neq 0.
\]
By~\cite[Remark~4.15]{PS08}, we may assume that the
pair~$(\RU_n,\psi_n)$ arises from the construction
of~\cite[Theorem~3.3]{PS08}. 
This construction, which produces a particular maximal
  unipotent subgroup and non-degenerate character, 
depends only on the simple character~$\theta_1$. Thus,
by~\cite[Corollary~4.13]{PS08}, the
space~$\Hom_{\RU_n\cap\bJ_1}(\psi_n,\Lambda)$ is non-zero for
\emph{any} extended maximal simple type~$(\bJ_1,\Lambda)$
containing~$\theta_1$.

Now let~$(\bJ_2,\Lambda_2)$ be any extended maximal simple type in~$\pi_2$. The hereditary
orders~$\FA_i$ have the same period~$e(\Theta)$ so are conjugate in~$\RG_n$;
replacing~$\FA_2$ by a conjugate if necessary, we assume they are equal.
Then the simple characters~$\theta_1,\theta_2$ are conjugate in~$\RG_n$, by definition
of endo-equivalence; again, replacing~$\theta_2$ by a conjugate if necessary, we assume
they are equal. Now~$\bJ_i$ is the~$\RG_n$-normalizer of~$\theta_i$ so we have~$\bJ_1=\bJ_2$.
Hence, by the remarks above,
\[
\Hom_{\RU_n\cap\bJ_1}(\psi_n,\Lambda_2)\neq 0.
\]
Thus the characters~$\Psi_n^1$,~$\Psi_n^2$ as defined in~\eqref{Psin} are equal.
Finally, by~\eqref{eqn:WsponPn}, the~$\bJ_1$-special Whittaker
functions~$W_{\pi_1}$,~$W_{\pi_2}$ defined by~\eqref{eqn:Wdef} agree on~$\RP_n$.
Thus~$(W_{\pi_1},W_{\pi_2})$ is a special pair for~$(\pi_1,\pi_2)$, which
completes the proof of Proposition~\ref{prop:endo-equivalent}.

\begin{rmk}\label{rmk:weak}
In the proof of the existence of a special pair, we do not in fact use that the
endo-classes for~$\pi_1,\pi_2$ coincide, but only that~$\bJ_1,\bJ_2$ are
contained in a common compact-modulo-center open subgroup of~$\RG$,
that~$H^1_1\cap\RP_n=H^1_2\cap\RP_n$ and that~$\theta_1,\theta_2$
coincide on~$H_1^1\cap\RP_n$. This is significantly weaker: for
example, if~$\deg(\pi_1)=n$ and~$\beta_1$ is a \emph{minimal} element
(see, for example,~\cite[Section~1.4]{BK93} for
the definition) then~$H_1^1\cap\RP_n=U^{\left[\frac m2\right]+1}(\FA_1)\cap\RP_n$.
\end{rmk}

\section{Conductors of pairs}

In this section we will prove Corollary~\ref{cor:nonmax}. The techniques here are
entirely different, relying on the explicit computation of conductors of pairs
of supercuspidal representations from~\cite{BHK98} and their application in~\cite{BH03}.

\subsection{Endo-classes}

In~\cite{BH03}, Bushnell and Henniart define a function~$\FF$ of pairs of endo-classes
with the property that, for~$\pi,\tau$ irreducible supercuspidal representations
of~$\RG_n,\RG_r$ respectively, with~$n>r\ge 1$, the conductor satisfies
\begin{equation}\label{eqn:fFF}
f(\pi\times\tiltau,\psi)=nr(\FF(\Theta(\pi),\Theta(\tau))+1).
\end{equation}
Moreover,~\cite[Theorem~C]{BH03} says that this function characterizes endo-classes
in the following way: for endo-classes~$\Theta_1,\Theta_2$,
\begin{equation}\label{eqn:FFchar}
\FF(\Theta_1,\Theta_2) \ge \FF(\Theta_1,\Theta_1),
\end{equation}
with equality if and only if~$\Theta_1=\Theta_2$.

The final ingredient we need is that, given an endo-class~$\Theta$, there is an
irreducible supercuspidal representation~$\tau$ of~$\RG_r$ with endo-class~$\Theta$
whenever~$r$ is a multiple of the degree~$\deg(\Theta)$. This is immediate from the
definitions of endo-class in~\cite{BH98} and of maximal simple types.

\subsection{Proof of Corollary~\ref{cor:nonmax}}

Let~$\pi_1,~\pi_2$ be unitarizable irreducible supercuspidal representations
of~$\RG_n$, with endo-classes~$\Theta_1,~\Theta_2$ respectively, and suppose
that~$\deg(\pi_1)<n$. Suppose the local gamma 
factors~$\gamma(s,\pi_1\times\tau,\psi)$ and~$\gamma(s,\pi_2\times\tau,\psi)$ are
equal as functions in the complex variable~$s$, for all irreducible supercuspidal
representations~$\tau$ of~$\RG_r$ with~$r=1,\ldots,\left[\frac{n}{2}\right]$.

Put~$r:=\deg(\pi_1)$, which is a proper divisor of~$n$; in particular,~$r\le\left[\frac n2\right]$.
Let~$\tau$ be an irreducible supercuspidal representation~$\tau$ of~$\RG_r$ with
endo-class~$\Theta(\tau)=\Theta_1$. Then, by~\eqref{eqn:fFF} and hypothesis, we have
\[
\FF(\Theta_1,\Theta_1) = \frac{f(\pi_1\times\tiltau,\psi)}{nr}-1
= \frac{f(\pi_2\times\tiltau,\psi)}{nr}-1 = \FF(\Theta_2,\Theta_1).
\]
We deduce, from~\eqref{eqn:FFchar}, that~$\Theta_1=\Theta_2$. Then
Proposition~\ref{prop:endo-equivalent} and Theorem~\ref{thm:sp} combine to
imply that~$\pi_1$ is equivalent to~$\pi_2$.

\begin{rmk}\label{rmk:end}
The restriction of a simple character~$\theta\in\CC(\FA,\beta,\psi)$ to the
groups~$H^{t+1}=H^1\cap\bU^{t+1}(\FA)$,~$t\ge 0$, determines a family
of endo-classes. By considering these endo-classes, rather than just
those coming from the simple character~$\theta$, it seems likely that
one could prove a more general version of Corollary~\ref{cor:nonmax}
by generalizing the function~$\FF$.

However, even in the most optimistic scenario, this will leave the
case where, for~$i=1,2$, we have~$\deg(\pi_i)=n$ and any simple
character~$\theta_i\in\CC(\FA,\beta_i,\psi)$ in~$\pi_i$ has~$\beta_i$
a minimal element. One would need to prove that the equality of
local gamma factors implies that one can
assume~$\beta_1\equiv\beta_2\pmod{\FP^{-\left[\frac m2\right]}}$ to enable us to
construct a special pair of Whittaker functions (see Remark~\ref{rmk:weak}).
However, even the case~$n=3$ seems to be very difficult to analyze directly
via the explicit construction of supercuspidal representations, even
in the tame case.
\end{rmk}

\bigskip
\footnotesize
\noindent\textit{Acknowledgments.}
The work of the first-named author is supported in part by NSF grants
DMS--1001672 and DMS--1301567. Both the first and the second-named
authors would like
to thank the Chinese Academy of Sciences for invitation and support
over the years when their research work related to this paper was in
process. The second-named author is supported by NSC101-2918-I-006-003
and would like to thank the University of East Anglia for providing a
stimulating environment during the execution period of the project.
The third-named author is supported by EPSRC grant EP/H00534X/1.
We thank the referee for some helpful comments, in particular for
encouraging us to remove an unnecessary hypothesis on the
characteristic.


\end{document}